\documentclass[10pt]{amsart}
\usepackage{amssymb, amsmath}
\usepackage{graphics}
\usepackage{latexsym}
\usepackage{amsmath}
\usepackage{amssymb,amsthm,amsfonts}
\usepackage{amscd}
\usepackage[arrow, matrix, curve]{xy}
\usepackage{syntonly}
\usepackage{yfonts}
\usepackage{tikz-cd}
\usepackage{filecontents}
\usepackage{mathtools}
\usepackage{stmaryrd}
\usepackage{MnSymbol}
\DeclareMathOperator{\spec}{Spec}

\DeclareMathOperator{\aut}{Aut}

\DeclareMathOperator{\pr}{pr}

\DeclareMathOperator{\jac}{Jac} 
\DeclareMathOperator{\pic}{Pic}
\DeclareMathOperator{\nm}{Nm}
\DeclareMathOperator{\deck}{Deck}
\DeclareMathOperator{\irr}{Irr}
\DeclareMathOperator{\cl}{Cliff} 
\DeclareMathOperator{\im}{Im}

\theoremstyle{plain}
\newtheorem{thm}{Theorem}[section]
\newtheorem{theorem}[thm]{Theorem}
\newtheorem{lemma}[thm]{Lemma}

\newtheorem{proposition}[thm]{Proposition}

\theoremstyle{definition}

\newtheorem{definition}[thm]{Definition}

\newtheorem{example}[thm]{Example}

\numberwithin{equation}{thm}

\newcommand{\sF}{{\mathcal F}}

\newcommand{\sO}{{\mathcal O}}
\newcommand{\sP}{{\mathcal P}}

\newcommand{\sT}{{\mathcal T}}

\newcommand{\Hom}{{\rm Hom}}

\begin{document}
\title{On the Prym map of Galois coverings}
\author{Abolfazl Mohajer}
\address{Universit\"{a}t Mainz, Fachbereich 08, Institut f\"ur Mathematik, 55099 Mainz, Germany}
\email{mohajer@uni-mainz.de} 
\maketitle
\begin{abstract}
In this paper we consider the Prym variety $P(\widetilde{C}/C)$ associated to a Galois coverings of curves $f:\widetilde{C}\to C$ branched at $r$ points. We discuss some properties and equivalent definitions and then consider the Prym map $\mathcal{P}=\mathcal{P}(G,g,r):R(G,g,r)\to A_{p,\delta}$ with $\delta$ the type of the polarization.  For Galois coverings whose Galois group is abelian and metabelian (non-abelian) we show that the differential of this map at certain points is injective. We also consider the Abel-Prym map $u:\widetilde{C}\to P(\widetilde{C}/C)$ and prove some results for its injectivity. In particular we show that in contrast to the classical and cyclic case, the behavior of this map here is more complicated. The theories of abelian and metabelian Galois coverings play a substantial role in our analysis and have been used extensively throughout the paper.
\end{abstract}

\section{Introduction}
To a given finite covering $f:\widetilde{C}\to C$ between non-singular projective algebraic curves (or Riemann surfaces) one can associate a so-called \emph{Prym variety}, a polarized abelian variety: $f$ induces a \emph{norm map}
\begin{align*}
\nm_f:\pic^0(\widetilde{C})\to \pic^0(C)\\
\sum a_ip_i\mapsto \sum a_if(p_i)
\end{align*}
The Prym variety associated to $f$ is then defined as
$P(f)=P(\widetilde{C}/C)=(\ker \nm_f)^0$, i.e., the connected
component of the kernel of $\nm_f$ containing the identity.
Identifying $\pic^0$ with the Jacobian, one sees that the
canonical (principal) polarization of $\jac(\widetilde{C})$
restricts to a polarization on $P(f)$.
\par Classically, $f$ is a double covering which is \'etale or
branched at exactly two points. In these cases, $P(f)$ is known to
be principally polarized. In fact, these are the only cases in
which the polarization on $P(f)$ is principal. However, the type
$\delta$ of the polarization on $P(f)$ depends on the topological
structure of the covering map $f$, see \cite{BL}. \par Let $G$ be
a finite group with $n=|G|$. Consider the following moduli stack
$R(G,g,r)$: Objects are couples $((C,x_1,\dots,
x_r),f:\widetilde{C}\to C)$ such that
\begin{enumerate}
\item $(C,x_1,\dots, x_r)$ is a smooth $r$-pointed curve of genus
$g$ i.e., a point in $M_g(r)$.
\item The group $G$ acts on the smooth curve $\widetilde{C}$ and
the map $f:\widetilde{C}\to \widetilde{C}/G=C$ is the quotient map
branched along the reduced divisor $D=\sum x_i$.
\end{enumerate}
Note that since our problem is insensitive to level structures, we
may actually consider $R(G,g,r)$ as a coarse moduli space. As a result, we omit any assumptions on the automorphism group of the base curve $C$ whose non-triviality can be remedied either by considering the moduli stack or by imposing level structures. In
section 2, we will describe this moduli space intrinsically in
terms of the curve $C$ using the theory of abelian and metabelian
coverings that will be explained in that section.\par Let $p$ be the dimension of the corresponding Prym variety and let
$A_{p,\delta}$ be denote the moduli space of abelian varieties
with polarization of type $\delta$ over $\mathbb{C}$. The above
constructions behave well also in the families of curves and hence
we obtain a morphism
\begin{equation}
\mathcal{P}=\mathcal{P}(G,g,r):R(G,g,r)\to A_{p,\delta}.
\end{equation}
We call the map $\mathcal{P}$ the \emph{Prym map of type
$(G,g,r)$}.
One of our objectives in this paper is to study this map. The Prym map is even in the classical case known to be non-injective which implies that one needs to stduy other closely related aspects, namely the generic injectivity. We therefore study the differential $d\mathcal{P}$ and examine at which points this differential is injective. This will be done in section 3 in which we also use the results of \cite{LO} by Lange and Ortega which partly motivated us to generalize its results to more general Galois coverings. Indeed the key points of the proof of Proposition \ref{dp} are direct generalizations of the results of \cite{LO} to the broader class of Galois covers. In section 3, we first describe the Prym varieties of Galois covers and give some equivalent definitions and prove some properties of this abelian variety. There is also the so-called Abel-Prym map which is induced from the classical Abel-Jacobi map. When the Galois group of the covering is cyclic, it is shown in \cite{LO} that the Abel-Prym map is non-injective at the ramification points of the covering provided that the curve $\widetilde{C}$ is not hyperelliptic. Here we generalize this under more restricitive conditions in particular that the curve $\widetilde{C}$ is not $g_{2^h}^{1}$. We provide an example to show that these conditions are indeed necessary, otherwise the corresponding statement will be false even for very simple covers. \par 
\section{Abelian and metabelian Galois coverings} \label{abeliancovers}
\subsection{Galois covers of curves}
Let us summarize some general facts about Galois coverings of
curves. Let $\widetilde{C},C$ be complex smooth projective
algebraic curves (equivalently Riemann surfaces) and let $f\colon
\widetilde{C}\to C$ be a Galois covering of degree $n$. By this we
mean precisely that there exists a finite group $G$ with $|G|=n$
together with a faithful action of $G$ on $\widetilde{C}$ such
that $f$ realizes $C$ as the quotient of $\widetilde{C}$ by $G$.
Consider the ramification and branch divisors $R,D$ of $f$. Note
that $R$ consists precisely of the points in $\widetilde{C}$ with
non-trivial stabilizers under the action of $G$. The deck
transformation group $\deck(\widetilde{C}/ C)$, i.e., the group of
those automorphisms of $\widetilde{C}$ that are compatible with
$f$ is isomorphic to the Galois group $G$ and acts transitively on
each fiber $f^{-1}(x)$. If $y\in \widetilde{C}$ is a ramification
point with ramification index $e$, then so are all points in the
fiber $f^{-1}(f(y))$. Moreover, the stabilizers of these points in
$\deck(\widetilde{C}/ C)\cong G$ are conjugate cyclic subgroups of
$\deck(\widetilde{C}/ C)$ (or by using the isomorphism
$\deck(\widetilde{C}/ C)\cong G$, subgroups of $G$) of order $e$,
see \cite{Sz}, Proposition 3.2.10. In particular the stabilizer of
a point in $\widetilde{C}$ is trivial, if and only if that point
is \emph{not} a ramification point. The stabilizer $H_y$ of a
point $y\in \widetilde{C}$ is also referred to as  the \emph{inertia subgroup}
of $y$. 
\subsection{Abelian Galois covers}\label{Abelian Galois covers} In this section  we describe the building data of an abelian cover and the construction of the
cover using these data and the relations among them. Notations
come mostly from \cite{P1} in which such Galois coverings
of algebraic varieties have been extensively studied. Let $f:\widetilde{C}\to C$ be a $G$-Galois cover with $G$ finite abelian branched above the points $x_1+\cdots+
x_r$. Since the group $G$ is abelian, the inertia group above a branch point $x_i$ is independent of the chosen ramification point and we denote it by $H_i$.  Then $f_*\mathcal{O}_{\widetilde{C}}=\bigoplus\limits_{\chi \in G^*}
L^{-1}_{\chi}$, where each $L_{\chi}$ is an invertible sheaf on
$C$ on which $G$ acts by character $\chi$. So in particular, the
invariant summand $L_1$ is ismorphic to $\mathcal{O}_C$. The
algebra structure on $f_*\mathcal{O}_{\widetilde{C}}$ is given by
the ($\mathcal{O}_C$-linear) multiplication rule
$m_{\chi,\chi^{\prime}}:L^{-1}_{\chi}\otimes
L^{-1}_{\chi^{\prime}}\to L^{-1}_{\chi\chi^{\prime}}$ and
compatible with the action of $G$. The choice of a primitive
$n$-th root of unity $\xi$ amounts to giving a map $\{1,\dots,
r\}\to G$, the image $h_i$ of $i$ under which is the generator of
the inertia group $H_i$ that is sent to $\xi^{n/n_i}$ by $\chi_i$. The line bundles
$L_{\chi}$ and divisors $x_i$ each labelled with an element $h_i$
as described above are called the building data of the cover.
These data are to satisfy the so-called \emph{fundamental relations} and determine the cover $f:\widetilde{C}\to C$ up to deck automorphisms. Let us write these relations down. For $i=1,\dots, r$ and $\chi\in G^*$, let $a^{i}_{\chi}$ be the smallest positive integer such that $\chi(h_i)=\xi^{na^{i}_{\chi}/n_i}$. For any two characters $\chi,
\chi^{\prime}$, $0\leq a^{i}_{\chi}+a^{i}_{\chi^{\prime}}<2n_i$,
so
\[\epsilon^{i}_{\chi,\chi^{\prime}}=[\frac{a^{i}_{\chi}+a^{i}_{\chi^{\prime}}}{n_i}]=
\begin{cases}
1& a^{i}_{\chi}+a^{i}_{\chi^{\prime}}\geq n_i\\
0, & a^{i}_{\chi}+a^{i}_{\chi^{\prime}}< n_i
\end{cases}\]
and we set $D_{\chi,\chi^{\prime}}=\sum\limits_{i=1}^{r}
\epsilon^{i}_{\chi,\chi^{\prime}}x_i$. By the fundamental
relations of the cover we mean the following isomorphisms:
\begin{align} \label{fundamentalrel}
\mu_{\chi, \chi^{\prime}}:L_{\chi}+L_{\chi^{\prime}}\xrightarrow{\sim}
L_{\chi\chi^{\prime}}\otimes \mathcal{O}_C(D_{\chi,
	\chi^{\prime}}).
\end{align}
In particular, if $\chi^{\prime}=\chi^{-1}$, then
\begin{align} \label{fundamentalrelinverse}
L_{\chi}+L_{\chi^{-1}}\equiv D_{\chi,\chi^{-1}},
\end{align}
and $D_{\chi,\chi^{-1}}$ the sum of the components $x_i$, where
$\chi(h_i)\neq 1$. The cover $f:\widetilde{C}\to C$ can be
recovered from the fundamental relations \ref{fundamentalrel} by
first defining the curve $\widetilde{C}$ inside the vector bundle
$\mathcal{L}=\oplus_{\chi\neq 1} L_{\chi}$ by the equations
\begin{equation} \label{fundamental eqs}
z_{\chi}z_{\chi^{\prime}}=(\prod_i
s_i^{\epsilon^{i}_{\chi,\chi^{\prime}}})z_{\chi\chi^{\prime}}
\end{equation}
where $z_{\chi}$ is the fiber coordinate of the bundle $L_{\chi}$
which can also be viewed as the tautological section of pull-back
of the bundle $L_{\chi}$ to $\mathcal{L}$ and $s_i\in
H^0(C,\mathcal{O}_C(x_i))$ is the (pull-back to $\mathcal{L}$ of
the) defining equation for $D_i$ for $i=1,\dots, r$. This is
naturally a flat $C$-scheme. Conversely, for every choice of the
sections $s_i$, equations \ref{fundamental eqs} define a flat
scheme $\widetilde{C}$ over $C$ which is smooth if and only if
each $D$ is reduced. We therefore have the following fundamental
theorem proven in \cite{P1}.
\begin{theorem} \label{structure abelian}
Let $G$ be a finite abelian group. Let $\widetilde{C}, C$ be
smooth projective algebraic curves and let $f\colon
\widetilde{C}\to C$ be an abelian cover with Galois group $G$.
With the notations as above, the following set of linear
equivalences hold.
\begin{equation} \label{fundamentalrel1}
\mu_{\chi, \chi^{\prime}}:L_{\chi}+L_{\chi^{\prime}}\xrightarrow{\sim} L_{\chi\chi^{\prime}}+ D_{\chi,\chi^{\prime}} \text{     } \forall \chi, \chi^{\prime}\in G^*.
\end{equation}
Conversely, given a set of data $\{L_{\chi}\}_{\chi \in G^*}, \{D_{\chi, \chi^{\prime}}\}$ consisting respectively of invertible sheaves and reduced effective divisors on $C$ satisfying the relation \ref{fundamentalrel1} determines an abelian cover. When $C$ is furthermore complete, this abelian cover is unique.
\end{theorem}
Now let the finite abelian group $G$ have the decomposition $G=\langle\sigma_1\rangle\oplus\langle\sigma_2\rangle\oplus\cdots\oplus \langle\sigma_h\rangle$. Let us denote by $G^*$ the character group of of $G$, called
also the dual abelian group. For $g\in G$, we denote by $g^*$ the
corresponding character of $G$.  We have naturally $G^*=\langle\sigma^*_1\rangle\oplus\langle\sigma^*_2\rangle\oplus\cdots\oplus \langle\sigma^*_h\rangle$. These data satisfy 
\begin{equation}\label{reduced building data relations}
L_i^{n_i}\cong\mathcal{O}_C(\sum_{j=1}^n\lambda_{ij}x_j).
\end{equation}
Denote $d_i=\sum_{j=1}^n\lambda_{ij}$. The collection $(\{L_i\}_{1\leq i\leq h}, \{D_i\}_{1\leq i\leq h})$ is called a \emph{reduced building data} for the abelian cover $f:\widetilde{C}\to C$, see \cite{P1}, Definition 2.3. 
\begin{theorem} \label{reduced structure theorem}
Let $C$ be a projective non-singular curve. The reduced building data 
\begin{equation}\label{reduced building data}
(\{L_i\}_{1\leq i\leq h}, \{D_i\}_{1\leq i\leq h})
\end{equation} 
determine the abelian cover $f:\widetilde{C}\to C$ uniquely up to isomorphisms of $G$-covers.
\end{theorem}
\subsection{Eigensheaves of the group action} \label{eigenspaces}
Let us describe the eigensheaves $L^{-1}_{\chi}$ in the
decomposition
$f_*\mathcal{O}_{\widetilde{C}}=\bigoplus\limits_{\chi \in G^*}
L^{-1}_{\chi}$. As before, $D$ denotes the branch divisor of $f$
with irreducible components $D_i$. We have already remarked that
the scheme $V$ can be constructed inside the (total space of the)
vector bundle $\mathcal{L}=\bigoplus\limits_{\chi\in
	G^*\setminus\{1\}} L^{-1}_{\chi}$ by the equations
\ref{fundamental eqs} in terms of the tautological section
$z_{\chi}$ of pull-back of the bundle $L_{\chi}$ to $\mathcal{L}$
and the defining equation $s_i\in H^0(C,\mathcal{O}_C(x_i))$ for
$x_i$. One can embed $C$ in $\mathcal{L}$ by the zero section of
$\mathcal{L}\to C$. Let the branch divisor $D$ be reduced. As a
closed subscheme, $C$ is given inside $\mathcal{L}$ by the
equations $z_{\chi}=0$. Let $R=f^{-1}(D)$. Let $p:\mathcal{L}\to
C$ be the bundle projection (we will use the same notation for
$\mathcal{L}$ and its total space). Suppose $H_j=\langle
h_j\rangle, j=1,\dots, s$ are the (non-trivial) inertia groups of
$f$ and let $R_j$ be the divisorial part of the reduced
ramification divisor $R_{red}$ consisting of all those points that
have $H_j$ as their stabilizer. Suppose the ramification index of
$R_j$ is $e_j$. Consider the group of characters $G^*$ and suppose
$\chi_j=h_j^*, j=1,\dots, s$ are the characters corresponding to
the $h_j$. Let $L_j=L_{\chi_j}$ be the line bundles associated to
the $\chi_j$ for $j=1,\dots, s$. It is clear that
$R_{red}=R_1+\cdots+R_s$ and that $R_j$ has its support in $L_j$.
Note that the equations \ref{fundamental eqs} imply that the
${\widetilde{C}}$ has at most singularities over the singular
points of the branch divisor $D$. 
\begin{proposition} \label{stack isom1}
The moduli stack $R(G,g,r)$ is isomorphic to the moduli stack of
the following objects
\[((C,x_1,\dots, x_r),\{L_{\chi}\}_{\chi \in G^*},
\{\mu_{\chi, \chi^{\prime}}\}_{\chi, \chi^{\prime} \in G^*}).\]
Where $(C,x_1,\dots, x_r)$ is a smooth $r$-pointed curve, the
$L_{\chi}$ are line bundle on $C$ for every $\chi \in G^*$ and the
$m_{\chi, \chi^{\prime}}$ are isomorphisms as in . The morphisms
are morphisms of $r$-pointed curves that induce maps of line
bundles compatible with the isomorphisms $\mu_{\chi,\chi^{\prime}}$.
\end{proposition}
Proposition \ref{stack isom1} has the following consequence:
Consider the forgetful map
\begin{gather}\label{forgetful}
R(G,g,r)\to M_g(r),\\
((C,x_1,\dots, x_r),\{L_{\chi}\}_{\chi \in G^*}, \{\mu_{\chi,
\chi^{\prime}}\}_{\chi, \chi^{\prime} \in G^*})\mapsto
(C,x_1,\dots, x_r)
\end{gather}
Proposition \ref{stack isom1} has the following consequence:
Consider the forgetful map
\begin{gather}\label{forgetful}
	R(G,g,r)\to M_g(r),\\
	((C,x_1,\dots, x_r),\{L_{\chi}\}_{\chi \in G^*}, \{\mu_{\chi,
		\chi^{\prime}}\}_{\chi, \chi^{\prime} \in G^*})\mapsto
	(C,x_1,\dots, x_r)
\end{gather}
This map exhibits $R(G,g,r)$ as a principal homogenous space (or
torsor) with the structure group $\Hom(G^*,\pic^0(C))$ at a given
point $((C,x_1,\dots, x_r),\{L_{\chi}\}_{\chi \in G^*}, \{m_{\chi,
	\chi^{\prime}}\}_{\chi, \chi^{\prime} \in G^*})\in R(G,g,r)$.
\subsection{Metabelian Galois covers} \label{Metabelian Galois covs}
In this subsection $G$ denotes a finite metabelian group. For a detailed treatment of metabelian covers we refer to \cite{M19} whose results and notations will be used throughout the paper. First of all, we have the following. First of all, we have the following.
\begin{definition} \label{defmetabelian}
A finite group $G$ is called metabelian (resp. metacyclic) if it
sits in the following short exact sequence of groups.
\begin{equation} \label{metabelian extension}
0\to A\to G\to N\to 0,
\end{equation}
where $A$ and $N$ are abelian (resp. cyclic). In other words a
group is metabelian (resp. metacyclic) if it has an a normal
subgroup $A$ such that both $A$ and $N=G/A$ are abelian (resp.
cyclic).
\end{definition}
Notice that the above definition is equivalent to saying that
metabelian groups are precisely the solvable groups of derived
length at most 2.\par It is straightforward to see that Definition
\ref{defmetabelian} implies that $G$ is metabelian if and only if
$G$ has the following presentation
\begin{equation} \label{presentation metabelian}
\langle \sigma_1,\dots,\sigma_s, \tau_1,\dots,\tau_l\mid
\sigma_i\sigma_j=\sigma_j\sigma_i, \tau_i\tau_j=\tau_j\tau_i,
\sigma_i^{m_i}=1, \sigma_i\tau_j=\tau_j\sigma^{r_{1ij}}_1\cdots
\sigma^{r_{sij}}_s, \tau_j^{a_j}=\sigma^{k_{1j}}_1\cdots
\sigma^{k_{sj}}_s\rangle
\end{equation}
Here $A=\langle \sigma_1,\dots,\sigma_s \rangle$ and
$N=\langle\overline{\tau}_1,\dots,\overline{\tau}_l \rangle$, and
$\overline{\tau}_j$ denotes the image of $\tau_j$ in $N=G/A$.\par
Now Let $G$ be a finite metabelian group as above, $X$ a smooth
algebraic curve (equivalently a Riemann surface) over $\mathbb{C}$
with $G\subset \aut(X)$ and $Y$ a smooth algebraic curve such that
$X/G=Y$ and the cover $\pi: X\to Y$ is Galois, i.e., $\pi$ is the
quotient map $X\to X/G$. The factorization \ref{metabelian extension} gives rise to a factorization, $\pi \colon
X\xrightarrow{p}Z\xrightarrow{q} Y$, where $p,q$ are the
corresponding intermediate abelian Galois covers, i.e., $p\colon
X\to Z=X/A$ is an abelian Galois covering with Galois group $A$
and $q\colon Z\to Y=Z/N$ is an abelian Galois covering with Galois
group $N$. Therefore to study the Galois covering $\pi \colon X\to
Y$, it is helpful to study these intermediate abelian coverings.
In \cite{M19}, the metabelian Galois coverings of algebraic
varieties have been analyzed using the theory of abelian Galois
coverings that we explained in Section \ref{abeliancovers}
developed in particular in \cite{P1}. Explicitely, by Theorem
\ref{structure abelian}, the cover $p\colon X\to Z$ is determined
by the existence of invertible sheaves $(\mathcal{F}_{\chi})_{\chi
	\in A^*}$, and reduced effective divisors $(D_i)$ without common
components on $Z$ such that \ref{fundamentalrel1} holds. Note the
multiplication map
$m_{\chi\chi^{\prime}}:\mathcal{F}_{\chi}\otimes
\mathcal{F}_{\chi{\prime}}\to \mathcal{F}_{\chi\chi^{\prime}}$.
Before stating one of the main structure theorems of \cite{M19},
let us introduce the following notation: Suppose $\chi$ is an
irreducible character of the abelian group $A=\langle
\sigma_1,\dots,\sigma_s \rangle$. Let $\tau_j\in N$. Since $A$ is
a normal subgroup of $G, \tau_j^{-1}\sigma_u\tau_j\in A$ for every
$u=1,\dots, s$. We define a new character $\chi^{(1)}_j$ of $A$
by $\chi^{(1)}_j(\sigma_u)=\chi(\tau_j^{-1}\sigma_u\tau_j)$ for
every $u=1,\dots, s$. Since $\chi$ is an irreducible character,
$\chi^{(1)}_j$ is also irreducible. In particular for each
$\gamma\in \mathbb{N}$ one can define a character
$\chi^{(\gamma)}_{j}$ of $A$ by setting
$\chi^{(\gamma)}_{j}(\sigma_u)=\chi(\tau_j^{-\gamma}\sigma_u\tau_j^{\gamma})$.
By presentation \ref{presentation metabelian}, it is clear that
$\chi^{(a_j)}_{j}=\chi.$ \par Now, we are ready to state our
structure theorem for metabelian Galois covers.
\begin{theorem} \label{structure of metabelian}(Structure theorem
for metabelian covers) A metabelian Galois cover $\pi \colon X\to
Y$ is determined by the following data:
\begin{enumerate}
	\item Line bundles $(L_{\eta})_{\eta\in N^*}$ and effective
	divisors $B_1,\dots, B_l$ on $Y$ such that
	$\phi_{\eta, \eta^{\prime}}:L_{\eta}+L_{\eta^{\prime}}\xrightarrow{\sim} \sO_C(\sum
	\epsilon^{i}_{\eta\eta^{\prime}} B_i)$. \item Reduced effective
	weil divisors $D_1,\dots, D_n$ on $Z=\spec (\oplus
	L_{\eta}^{-1})$ identifying the character $\chi_i$ with $i$,
	$\overline{\tau_j}(D_i)=D_{\chi^{(1)}_{ij}}$, where
	$\chi^{(1)}_{ij}$ is the character of $A$ associated to $\chi_i$
	defined above. \item Invertible sheaves
	$\mathcal{F}_{\chi_1},\dots, \mathcal{F}_{\chi_n}$ on $Z$ such
	that the linear equivalence \ref{fundamentalrel1} holds and for
	every $\gamma\in \mathbb{N}$,
	$\overline{\tau_j}^{\gamma}(\mathcal{F}_{\chi_i})=\mathcal{F}_{\chi^{(\gamma)}_{ij}}$,
	where $\chi^{(\gamma)}_{ij}$ is defined above. Furthermore,
	$\overline{\tau_j}^{a_j}$ acts on the local sections of
	$\mathcal{F}_{\chi_i}$ as multiplication by
	$\exp(\frac{2\pi\sqrt{-1}\delta_{ij}}{m_i})$.
\end{enumerate}
\end{theorem}
For a proof of this theorem see \cite{M19}, Theorem 3.2. With the
notations of the beginning of \S \ref{eigenspaces}, let $L^{\prime}=\otimes_{j=1}^s L_j^{e_j-1}$ and let $\pi \colon X\xrightarrow{p}Z\xrightarrow{q} Y$ be the factorization  introduced in \S \ref{Metabelian Galois covs}. Then \cite{M19}, 3.3.1 gives:
\begin{equation}\label{metabelian invariant}
\begin{aligned}
\pi_*\omega_X=q_*(p_*\omega_X)=q_*(\omega_Z\otimes
p_*\mathcal{O}_X) =q_*(q^*(\omega_Y\otimes L^{\prime}\otimes (\oplus \mathcal{F}_{\chi}))= \\
q_*((\oplus \mathcal{F}_{\chi})\otimes q^*(\omega_Y\otimes
L^{\prime}))=q_*(\oplus \mathcal{F}_{\chi})\otimes \omega_Y\otimes L^{\prime}=\\
(\oplus U_{\chi})\otimes \omega_Y\otimes L^{\prime}=\oplus
(\omega_Y\otimes L^{\prime}\otimes U_{\chi})
\end{aligned}
\end{equation}
Where $U_i=q_*(\mathcal{F}_{\chi})$. As for the abelian case, Theorem \ref{structure of metabelian} yields the following isomorphism of the stack $R(G,g,r)$.
\begin{proposition} \label{stack isom2}
The moduli stack $R(G,g,r)$ is isomorphic to the moduli stack of the following objects
\[((C,x_1,\dots, x_r),\{L_{\eta}\}_{\eta \in N^*},
\{\phi_{\eta, \eta^{\prime}}\}_{\eta, \eta^{\prime} \in N^*}, \{\mathcal{F}_{\chi}\}_{\chi \in A^*}, \{\mu_{\chi, \chi^{\prime}}\}_{\chi, \chi^{\prime} \in A^*}).\]
Where $(C,x_1,\dots, x_r)$ is a smooth $r$-pointed curve, the $L_{\eta}$ are line bundle on $C$ for every $\eta \in N^*$ and the $\phi_{\eta, \eta^{\prime}}$ are isomorphisms as in \ref{fundamentalrel1}. The $\mathcal{F}_{\chi}$ are line bundle on $C$ for every $\chi \in A^*$ and the $\mu_{\chi, \chi^{\prime}}$ are isomorphisms. The morphisms are morphisms of $r$-pointed curves that induce maps of line bundles compatible with the isomorphisms $\mu_{\chi,
\chi^{\prime}}$.
\end{proposition}
\section{Prym varieties and the Prym map}
\subsection{Generalities} \label{generalities}
In general, one can associate a Prym variety to a subtorus of an
abelian variety $A$. We follow \cite{RR}. Let $A$ be a principally
polarized abelian variety over $\mathbb{C}$ and let
$X\xrightarrow{i} A$ be a subtorus. We denote by
$\lambda:A\to \widehat{A}$ the principal polarization of $A$. The
Prym variety of $X$ in $A$ is defined as
$P=P(A,\lambda,X)=\lambda^{-1}(\ker\widehat{i})$. \par Now let
$f:\widetilde{C}\to C$ be a covering map between smooth algebraic
curves as in introduction. \par Let us denote the Jacobians of the
curves $\widetilde{C}$ and $C$ respectively by $\widetilde{J}$ and
$J$. Note that by definition, if $R$ is a Riemann surface,
\[J(R)=\jac(R)=H^0(R,\omega_{R})^*/H_1(R,\mathbb{Z}).\]
Since the finite group $G$ acts on $\widetilde{C}$ it also acts on
the space of differential 1-forms
$H^0(\widetilde{C},\omega_{\widetilde{C}})$ and
$H_1(\widetilde{C},\mathbb{Z})$ and hence on the Jacobian
$\widetilde{J}$. In particular, we denote by $\widetilde{J}^G$ the
subgroup of fixed points of $\widetilde{J}$ under the action of
$G$. The following theorem is proven in \cite{RR} (repectively,
Theorem 2.5 and Proposition 3.1).
\begin{theorem} \label{Subtorus-jacobian}
\begin{enumerate}
\item $P=P(\widetilde{C}/C)$ is the Prym variety of the abelian subvariety $f^*J$ of the principally polarized abelian variety $\widetilde{J}$, i.e., $P=P(\widetilde{C}/C)=P(\widetilde{J},\widetilde{\lambda},f^*J)$.
\item $f^*J=(\widetilde{J}^G)^0$.\item The map $f$ induces an isogeny $J\times P(\widetilde{C}/C)\sim \widetilde{J}$
\end{enumerate}
\end{theorem}
We note that the isogeny mentioned in Theorem \ref{Subtorus-jacobian} is given by
\begin{gather}\label{isogeny equation}
\phi:J\times P(\widetilde{C}/C)\to \widetilde{J}\nonumber\\
\phi(c,\widetilde{c})=f^*c+\widetilde{c}
\end{gather}
By the above mentioned $G$-action on $H^0(\widetilde{C},\omega_{\widetilde{C}})$ and
$H_1(\widetilde{C},\mathbb{Z})$,  we set:
\begin{equation} \label{plus minus}
H^0(\widetilde{C},\omega_{\widetilde{C}})^+=H^0(\widetilde{C},\omega_{\widetilde{C}})^G(\cong H^0(C,\omega_{C})) \text{ and
}H^0(\widetilde{C},\omega_{\widetilde{C}})^-=H^0(\widetilde{C},\omega_{\widetilde{C}})/H^0(\widetilde{C},\omega_{\widetilde{C}})^+= \bigoplus\limits_{\chi \in \irr(G)\setminus\{1\}}H^0(\widetilde{C},\omega_{\widetilde{C}})^{\chi}
\end{equation}
Notice that
$H^0(\widetilde{C},\omega_{\widetilde{C}})=H^0(\widetilde{C},\omega_{\widetilde{C}})^+\oplus
H^0(\widetilde{C},\omega_{\widetilde{C}})^-$. \par The following
lemma is then an immediate consequence of Theorem
\ref{Subtorus-jacobian} above.
\begin{lemma} \label{prymvar}
Let $f:\widetilde{C}\to C$ be a Galois covering, then 
\begin{equation} \label{Prymdef}
P(\widetilde{C}/C)={H^0(\widetilde{C},\omega_{\widetilde{C}})^-}^*/H_1(\widetilde{C},\mathbb{Z})^-
\end{equation}
\end{lemma}
For a Galois covering $f:\widetilde{C}\to C$ of a curve $C$ of genus $g$ as above, one can compute the genus $\widetilde{g}\coloneqq g(\widetilde{C})$ by the Riemann-Hurwitz formula.
Using the isogeny $f^*J\times P(\widetilde{C}/C)\sim \widetilde{J}$ we see that the dimension of the Prym variety $P(\widetilde{C}/C)=P(f)$ is equal to $p=\widetilde{g}-g$.
The canonical principal polarization on $\widetilde{J}$ restricts to a polarization of type $\delta=(1,\dots,1,n,\dots,n)$ where $1$ occurs $p-(g-1)$ times and $n$ occurs $g-1$ times if $r=0$ and $1$ occurs $p-g$ times and $n$ occurs $g$ times otherwise.\par The canonical quotient map $H^0(\widetilde{C},\omega_{\widetilde{C}})\to H^0(\widetilde{C},\omega_{\widetilde{C}})^-=H^0(\widetilde{C},\omega_{\widetilde{C}})/H^0(\widetilde{C},\omega_{\widetilde{C}})^+$ induces a map $\Phi:\widetilde{J}\to P(\widetilde{C}/C)$. In fact this map induces the isogeny $\widetilde{J}\sim J\times P(\widetilde{C}/C)$ mentioned in Theorem \ref{isogeny equation}. Let $\widetilde{u}:\widetilde{C}\to \widetilde{J}$ be the Abel-Jacobi map for the curve $\widetilde{C}$. We have the following commutative diagram.
\begin{equation} \label{Abel-Jacobi}
\begin{tikzcd}
\widetilde{C} \arrow{rr}{u} \arrow[swap]{dr}{\widetilde{u}} & & P(\widetilde{C}/C) \\
& \widetilde{J} \arrow[swap]{ur}{\Phi}
\end{tikzcd}
\end{equation}
In analogy with the Abel-Jacobi map, we call the map $u:\widetilde{C}\to P(\widetilde{C}/C)$ \emph{the Abel-Prym} map of the covering $f$. We summarize some of the consequences of the above discussions in
\begin{lemma} \label{tangent}
Let $A=V/\Lambda$ be a complex abelian variety. Denote by $\Omega_{A,0}$ (resp. $T_{A,0}$) the cotangent space (resp. tangent space) of $A$ at the origin and by $\sT_{A}$ be the tangent bundle of the abelian variety $A$. The rest of the notations be as in Theorem \ref{Subtorus-jacobian}. Then it holds 
\begin{enumerate}
\item $\Omega_{A,0}\cong V^*$ (and hence $T_{A,0}\cong V$). In particular for the abelian variety $P\coloneqq P(\widetilde{C}/C)$ we have $\Omega_{P,0}=H^0(\widetilde{C},\omega_{\widetilde{C}})^-$ and $T_{P,0}=H^1(\widetilde{C},\sO_{\widetilde{C}})^-$.
\item $\sT_{A}\cong T_{A,0}\otimes \sO_{A}\cong V\otimes \sO_{A}$ and $H^1(A,\sT_{A})\cong H^1(A,\sO_{A})^{\otimes 2}\cong V^{\otimes 2}$. In particular, $\sT_{\widetilde{J}}\cong T_{\widetilde{J},0}\otimes \sO_{\widetilde{J}}$ and $H^1(\widetilde{J},\sT_{\widetilde{J}})\cong H^1(\widetilde{C},\sO_{\widetilde{C}})^{\otimes 2}$
and furthermore $H^1(P,\sT_{P})\cong {H^1(\widetilde{C},\sO_{\widetilde{C}})^-}^{\otimes 2}$ and $H^1(\widetilde{C},u^*\sT_{P})\cong H^1(\widetilde{C},\sO_{\widetilde{C}})^-\otimes H^1(\widetilde{C},\sO_{\widetilde{C}})$. 
\end{enumerate}
\end{lemma}
\begin{proof}
The first assertion of (1) is proved in \cite{BL} and the second one follows from this together with \ref{Prymdef} and the Serre duality. \par The first assertion of (2), namely the triviality of the tangent bundle of an abelian variety is also well-known, see \cite{BL},  \S 1.4 and the isomorphism of $H^1(A,\sT_{A})$ follows from this and the Serre duality again. The rest of the isomorphisms in (2) for the cohomology of the tangent bundle of the Jacobian and the Prym variety follow from the first part of (2). 
\end{proof}
Let
\begin{equation}
\mathcal{P}=\mathcal{P}(G,g,r):R(G,g,r)\to A_{p,\delta}
\end{equation}
be the Prym map (of type $(G,g,r)$) associated to the above families as in introduction. In the sequel, we would like to compute the (co)differential of the map $\mathcal{P}$ at a given point $((C,x_1,\dots, x_r),f:\widetilde{C}\to C)\in R(G,g,r)$. Therefore, we first explain the general set-up for this problem and in later sections do the computations in some special cases. Using the Proposition \ref{stack isom2}, one sees that the forgetful map \ref{forgetful} is a principal homogeneous space over $M_g(r)$. The tangent space to $R(G,g,r)$ at a point $((C,x_1,\dots, x_r),f:\widetilde{C}\to C)$ is isomorphic to $H^1(C,\sT_C(-D))$, where $\sT_C$ denotes the tangent bundle of the curve $C$, see \cite{Ser}, \S 3.4.3 in particular Example 3.4.14. Note that there is an isomorphism $H^1(C,\sT_C(-D))^{\vee}\cong  H^0(C,\omega^{\otimes 2}_C(D))$. The cotangent space to $A_{p,\delta}$ at the point $P\coloneqq P(\widetilde{C}/C)$ is isomorphic to $S^2(H^0(P,\Omega^1_P))\cong S^2(H^0(\widetilde{C},\omega_{\widetilde{C}})^-)$. By \ref{plus minus}, we have
\begin{equation} \label{symmetric square}
S^2(H^0(\widetilde{C},\omega_{\widetilde{C}})^-)\cong \bigoplus\limits_{\chi \in \irr(G)\setminus\{1\}}S^2(H^0(\widetilde{C},\omega_{\widetilde{C}})^{\chi})\oplus \bigoplus\limits_{\chi,\eta \in \irr(G)\setminus\{1\}}H^0(\widetilde{C},\omega_{\widetilde{C}})^{\chi}\otimes H^0(\widetilde{C},\omega_{\widetilde{C}})^{\eta},
\end{equation}
where $\irr(G)$ denotes the set of irreducible representations of $G$. On the other hand, the action of the group $G$ on $H^0(\widetilde{C},\omega_{\widetilde{C}})$ induces a natural $G$-action on the space $S^2(H^0(\widetilde{C},\omega_{\widetilde{C}})^-)$. Let $\psi\in \irr(G)$ be an irreducible character of $G$ with the corresponding representation $\rho_{\psi}$. The eigenspace $S^2(H^0(\widetilde{C},\omega_{\widetilde{C}})^-)^\psi$ of $S^2(H^0(\widetilde{C},\omega_{\widetilde{C}})^-)$ corresponding to $\psi$ is
\begin{equation} \label{eigenspce of sym}
S^2(H^0(\widetilde{C},\omega_{\widetilde{C}})^-)^\psi= \bigoplus\limits_{\substack{\chi\in \irr(G)\setminus\{1\}\\ \chi^2=\psi}}S^2(H^0(\widetilde{C},\omega_{\widetilde{C}})^{\chi})\oplus\bigoplus\limits_{\substack{\chi\in \irr(G)\setminus\{1\}\\ \chi \eta=\psi}}H^0(\widetilde{C},\omega_{\widetilde{C}})^{\chi}\otimes H^0(\widetilde{C},\omega_{\widetilde{C}})^{\eta}
\end{equation}
One obtains the following commutative diagram
\begin{equation} \label{normal diag}
\begin{tikzcd} 
S^2(H^0(\widetilde{C},\omega_{\widetilde{C}})^-) \arrow{r}{} \arrow[swap]{d}{p^+_1} & H^0(\widetilde{C},\omega^2_{\widetilde{C}})^{\psi} \arrow{d}{p^+_2} \\
\bigoplus H^0(\widetilde{C},\omega_{\widetilde{C}})^{\chi}\otimes H^0(\widetilde{C},\omega_{\widetilde{C}})^{\chi^{-1}} \arrow{r}{\mu} & H^0(\widetilde{C},\omega^2_{\widetilde{C}})^{G},
\end{tikzcd}
\end{equation}
in which the top horizontal arrow is given by 
\[S^2(H^0(\widetilde{C},\omega_{\widetilde{C}})^-)\to S^2(H^0(\widetilde{C},\omega_{\widetilde{C}})^-)^\psi\xrightarrow{\pr_2} \bigoplus\limits_{\substack{\chi\in \irr(G)\setminus\{1\}\\ \chi \eta=\psi}}H^0(\widetilde{C},\omega_{\widetilde{C}})^{\chi}\otimes H^0(\widetilde{C},\omega_{\widetilde{C}})^{\eta}\xrightarrow{m}H^0(\widetilde{C},\omega_{\widetilde{C}})^{\psi} \]
where $\pr_2$ is the projection to the second sum in \ref{eigenspce of sym} and $m$ is the multiplication of differential forms, vertical arrows are natural projections in \ref{eigenspce of sym} to the $G$-invariant subspaces and the bottom horizontal arrow is again the multiplication of differential forms. 
\subsection*{The Prym map}
Let $R(G,g,r)$ be as in introduction with $|G|=n$. By the definition of the Prym variety $P$, the tangent space $T_P$ of $P$ at the origin is
\begin{equation}\label{abelian tangent space}
T_P=(T_{J(\widetilde{C})})^-=(H^0(\widetilde{C},\omega_{\widetilde{C}})^-)^*=H^1(\widetilde{C},\mathcal{O}_{\widetilde{C}})^-.
\end{equation}
The equation \ref{abelian tangent space} implies that the cotangent space $T^*_P$ of $P$ at the origin is
\begin{equation}\label{abelian cotangent space}
T^*_P=H^0(\widetilde{C},\omega_{\widetilde{C}})^-
\end{equation}
As explained earlier, the tangent space of $R(G,g,r)$ at a point $((C,x_1,\dots, x_r),f:\widetilde{C}\to C)$ is $H^1(C,\mathcal{T}_C(-B))\cong H^0(C,\omega_C^2(B))$, see
\cite{Ser}, \S 3.4.3, especially Example 3.4.14. The isomorphism is the Serre duality.
We have
\begin{proposition} \label{dp}
	The codifferential $d\mathcal{P}$ of the Prym map can be identified with the canonical map
	\[\varphi:S^2(H^0(\widetilde{C},\omega_{\widetilde{C}})^-)\to H^0(C,\omega_C^2(B)),\]
	which is the composition $\mu\circ p$ of \ref{normal diag}.
\end{proposition}
In order to prove the above result, we will need the following lemma, which is \cite{LO}, Lemma 4.2.
\begin{lemma} \label{curve to ab.var}
	Let $X$ be a smooth projective curve and $A=V/\Lambda$ an abelian variety. If $u:X\to A$ is a non-constant morphism whose image generated $A$, then the dual of the differential
	\[H^1(du):H^1(X,\mathcal{T}_X)\to H^1(X,u^*\mathcal{T}_A)\]
	coincides with the multiplication of sections
	\[V^*\otimes H^0(X,\omega_X)\to H^0(X,\omega^{\otimes 2}_X)\]
\end{lemma}
Recall from \ref{Abel-Jacobi} that the Abel-Prym map $u:\widetilde{C}\to P$
factors through Abel-Jacobi map $u:J(\widetilde{C})\to P$. Moreover, we have an isogeny $t:P\to\widehat{P}$ induced by the polarization of $P$. This induces an isomorphism $T_P\cong T_{\widehat{P}}$. Therefore we can identify $T_P$ with
$H^1(P,\mathcal{O}_P)$ and by \ref{prymdef} with $H^1(\widetilde{C},\mathcal{O}_{\widetilde{C}})^-$. Moreover, by the fact that the tangent bundle of the abelian variety $J(\widetilde{C})$ is trivial and its fiber over $0$ is
$T_{\widetilde{C}}$ one obtains $\mathcal{T}_{\widetilde{C}}\cong T_{{\widetilde{C}}}\otimes\mathcal{O}_{J(\widetilde{C})}$ and this induces an isomorphism
$H^1(J(\widetilde{C}),\mathcal{T}_{J(\widetilde{C})})=H^1(\widetilde{C},\mathcal{O}_{\widetilde{C}})^{\otimes 2}$. Furthermore, the isomorphism $\mathcal{T}_P\cong
T_{\widehat{P}}\otimes \mathcal{O}_P$ gives the identifications
$H^1(P,\mathcal{T}_P)=H^1(\widetilde{C},\mathcal{O}_{\widetilde{C}})^-\otimes
H^1(P,\mathcal{O}_P)={H^1(\widetilde{C},\mathcal{O}_{\widetilde{C}})^-}^{\otimes 2}$
and
$H^1(\widetilde{C},u^*\mathcal{T}_{P})=H^1(\widetilde{C},\mathcal{O}_{\widetilde{C}})^-\otimes
H^1(\widetilde{C},\mathcal{O}_{\widetilde{C}})$. We also need the following lemma, which is proved in \cite{LO}, Lemma 4.3 and whose proof relies merely on the fact that the Abel-Prym map is a projection of the Abel-Jacobi map as indicated in \ref{Abel-Jacobi}
together with the above identities.
\begin{equation} \label{differential AP}
H^1(du):H^1(\widetilde{C},\mathcal{O}_{\widetilde{C}})\xrightarrow{H^1(d\widetilde{u})}
H^1(J(\widetilde{C}),\mathcal{T}_{J(\widetilde{C})})\xrightarrow{p^-\otimes p^-}H^1(P,\mathcal{T}_P)\xrightarrow{u^*}H^1(\widetilde{C},u^*\mathcal{T}_{P})
\end{equation}
where $p^-$ is the projection to the subspace $H^1(\widetilde{C},\mathcal{O}_{\widetilde{C}})^-$ (as in \ref{Abel-Jacobi}). 
\begin{proof} (of Proposition \ref{dp})
	Let $((C,x_1,\dots, x_r),f:\widetilde{C}\to C)$ be a point of
	$R(G,g,r)$. The tangent space of $A_{p,\delta}$ at the point $P$
	is equal to $S^2T_P=S^2(H^1(\widetilde{C},\mathcal{O}_{\widetilde{C}})^-)$.
	The product $H^1(\widetilde{C},\mathcal{T}_{\widetilde{C}})\times
	H^0(\widetilde{C},\omega_{\widetilde{C}})\to
	H^1(\widetilde{C},\mathcal{O}_{\widetilde{C}})$ respects the group
	action and hence induces
	$H^1(\widetilde{C},\mathcal{T}_{\widetilde{C}})^+\times
	H^0(\widetilde{C},\omega_{\widetilde{C}})^-\to
	H^1(\widetilde{C},\mathcal{O}_{\widetilde{C}})^-$. The induced map
	$H^1(\widetilde{C},\mathcal{T}_{\widetilde{C}})^+\to
	{H^0(\widetilde{C},\omega_{\widetilde{C}})^-}^*\otimes
	H^1(\widetilde{C},\mathcal{O}_{\widetilde{C}})^-$ is symmetric so
	that we get a map
	$H^1(\widetilde{C},\mathcal{T}_{\widetilde{C}})^+\to
	S^2(H^1(\widetilde{C},\mathcal{O}_{\widetilde{C}})^-)$. This is
	the differential of the Prym map $\mathcal{P}:R(G,g,r)\to
	A_{p,\delta}$ at the point $((C,x_1,\dots, x_r),f:\widetilde{C}\to
	C)$. Now \ref{differential AP} implies that this map can be considered as a map
	$H^1(\widetilde{C},\mathcal{T}_{\widetilde{C}})^+\to
	H^1(\widetilde{C},u^*\mathcal{T}_{P})$ whose image lies in
	$S^2(H^1(\widetilde{C},\mathcal{O}_{\widetilde{C}})^-)$. Note the identities shown above. Lemma \ref{curve to ab.var} then shows that the differential $d\mathcal{P}$ at the point $((C,x_1,\dots, x_r),f:\widetilde{C}\to C)$ is the multiplication map $S^2(H^0(\widetilde{C},\omega_{\widetilde{C}})^-)\to H^0(C,\omega^{\otimes 2}_{C}(B))^-$. 	
\end{proof}
\subsection{Prym varieties of abelian covers}
In this section explain the constructions in section \ref{generalities} for an abelian group $G$ based on the constructions of section \ref{Abelian Galois covers}. So let $f:\widetilde{C}\to C$ be a $G$-Galois cover of $C$, with $G$ a finite abelian group. We have
\begin{align}
H^0(\widetilde{C},\omega_{\widetilde{C}})=H^0(C,f_*\omega_{\widetilde{C}})=H^0(C,\oplus(\omega_{C}\otimes
L_{\chi^{-1}}))=\oplus_{\chi\in G^*} H^0(C,\omega_{C}\otimes
L_{\chi^{-1}})
\end{align}
where the second equality is due to the equality
$(f_*\omega_{\widetilde{C}})^{\chi}=\omega_{C}\otimes
L_{\chi^{-1}}$ for abelian covers, see \cite{P1}. In view of the above equalities, one obtains
\begin{align}  
H^0(\widetilde{C},\omega_{\widetilde{C}})^+=H^0(C,\omega_{C})\\
H^0(\widetilde{C},\omega_{\widetilde{C}})^-=\oplus_{\chi\in G^*\setminus\{1\}} 
H^0(\widetilde{C},\omega_{\widetilde{C}})^{\chi})=\oplus_{\chi\in
G^*\setminus\{1\}} H^0(C,\omega_{C}\otimes L_{\chi^{-1}})  \label{abelian decomposition}
\end{align}
So $P=P(\widetilde{C}/C)=\oplus_{\chi\in G^*\setminus\{1\}} H^0(C,\omega_{C}\otimes L_{\chi^{-1}})/\oplus_{\chi\in G^*\setminus\{1\}} H_1(\widetilde{C},\mathbb{Z})^{\chi}$ by virtue of \ref{abelian decomposition} and Lemma \ref{prymvar}. \par
 In particular, in the abelian case the multiplication map $\mu$ in the diagram \ref{normal diag} takes the following form.
\begin{equation}  \label{multiplication abelian}
\mu:\oplus_{\chi\in G^*} H^0(C,\omega_{C}\otimes L_{\chi})\otimes H^0(C,\omega_{C}\otimes L_{\chi^{-1}})\to H^0(C,\omega^{\otimes 2}_{C}(D_{\chi,\chi^{-1}})).
\end{equation}  
\subsection*{The Abel-Prym map}Let $G$ be a finite abelian group such that $G=\langle\sigma_1\rangle\oplus\langle\sigma_2\rangle\oplus\cdots\oplus \langle\sigma_h\rangle$. In other words $\sigma_1,\dots,\sigma_h$ are independent generators of $G$. Let $P\coloneqq P(\widetilde{C}/C)$ be the Prym variety associated
to the $G$-Galois cover $f:\widetilde{C}\to C$ with the Abel-Prym
map $u:\widetilde{C}\to P$. In this case the map $\Phi$ in \ref{Abel-Jacobi}
is the map $(1-\sigma_1)\cdots(1-\sigma_h)$ and so
$P=\im((1-\sigma_1)\cdots(1-\sigma_h))$. Suppose furthermore that
$ord(\sigma_i)>2$ for every $i$.
\begin{proposition}\label{injective AJ}
Suppose the curve $\widetilde{C}$ is \emph{not} a $g_{2^h}^{1}$. Then $u(p)=u(q)$ if and only if $p$ and $q$ are ramification points of $f$. In particular, if in addition $f$ is unramified, then $u$ is injective.
\end{proposition}
\begin{proof}
Let $c\in\widetilde{C}$ be an arbitrary base point giving the Abel-Jacobi map $\widetilde{u}:\widetilde{C}\to \widetilde{J} ,p\mapsto p-c$. By definition $u(p)=u(q)$ if and only if
\begin{equation}
(1-\sigma_1)\cdots(1-\sigma_h)(p-c)\sim(1-\sigma_1)\cdots(1-\sigma_h)(q-c)
\end{equation}
Or equivalently $(1-\sigma_1)\cdots(1-\sigma_h)(p)\sim(1-\sigma_1)\cdots(1-\sigma_h)(q)$.
After expanding and rearranging this, so that both sides are effective divisors, the assumption that $\widetilde{C}$ is not a $g_{2^h}^{1}$ implies that $p$ (resp. $q$) satisfies an equation of the form $\sigma_{i_1}^{r_i}\sigma_{i_2}^{r_2}\cdots\sigma_{i_t}^{r_t}(p)=p$
with $r_j=1$ or $2$. Since $ord(\sigma_i)>2$ and $\sigma_i$ are independent generators by assumption, it follows that the above equation is non-trivial and hence $p$ (resp. $q$) is a ramification point.
\end{proof}
The following example shows that if the assumptions of Proposition \ref{injective AJ} are not satisfied, then its result will no longer be valid even for fairly simple covers.
\begin{example}\label{counterexample AJ}
Suppose $f:\widetilde{C}\to C$ is a Galois cover with the Galois group $G=\langle\sigma_1\rangle\oplus\langle\sigma_2\rangle\cong\mathbb{Z}_2\times\mathbb{Z}_2$
(so $ord(\sigma_i)=2$). Suppose $p\in\widetilde{C}$ is \emph{not} a ramification point. Viewing $\sigma_1,\sigma_2$ as deck automorphisms of the curve $\widetilde{C}$, set
$q=\sigma_1\sigma_2(p)$. Then $q$ is also not a ramification point and $(1-\sigma_1)(1-\sigma_2)(p)=(1-\sigma_1)(1-\sigma_2)(q)$ as divisors so that $u(p)=u(q)$.
\end{example}
\subsection*{Injectivity of the differential}
By Proposition \ref{dp}, the injectivity of $d\mathcal{P}$ at a point $((C,x_1,\dots, x_r),f:\widetilde{C}\to C)$ is equivalent to the surjectivity of the multiplication map
\[\mu:H^0(C,\omega_{C}\otimes L_{\chi})\otimes H^0(C,\omega_{C}\otimes L_{\chi^{-1}})\to H^0(C,\omega^{\otimes 2}_{C}(D_{\chi,\chi^{-1}})),\]
for every character $\chi\in G^*$, where $D_{\chi,\chi^{-1}}$ is
the divisor in 2.1.2.\par We first consider the case of \'etale coverings ($r=0$). In this
case we abbreviate $R(G,g,0)$ by $R(G,g)$ which is the moduli
space of unramified $G$-Galois coverings of curves of genus $g$.
\begin{proposition}
Let $C$ be a smooth projective curve and $G$ a finite abelian group with $|G|=n$. In the following cases the differential $d\mathcal{P}$ of the Prym map $\mathcal{P}:R(G,g)\to
A_{p,\delta}$ at a given point $((C,x_1,\dots,x_r),f:\widetilde{C}\to C)$ is injective.
\begin{enumerate}
\item If $n$ is an even number and $\cl(C)\geq 3$. 
\item If $\cl(C)\geq 2n-1$.
\end{enumerate}
\end{proposition}
\begin{proof}
Take a reduced building data as in \ref{reduced building data relations}. Since by assumption, the covering is unramified, these relations are of the form $L_i^{n_i}\cong\mathcal{O}_C$. So $\deg L_i=0$, and the Riemann-Roch together with \cite{LO}, Corollary 2.3 implies  that both $\omega\otimes L_i$ and 
$\omega\otimes L_i^{-1}$ are very ample. Finally, \cite{Bu}, Theorem 1 shows that
\[\mu:H^0(C,\omega\otimes L_i)\otimes H^0(C,\omega\otimes L_i^{-1})\to H^0(C,\omega^{\otimes 2}_{C}),\]
is surjective and this gives desired result.
\end{proof}
\begin{proposition}
For $n\geq 2$ and $g\geq 7$ the Prym map $\mathcal{P}:R(G,g)\to A_{p,\delta}$ is generically finite.
\end{proposition}
\begin{proof}
Given a point $(C,f:\widetilde{C}\to C)\in R(G,g)$, it suffices to show that the differential $d\mathcal{P}$ is injective at $(C,f:\widetilde{C}\to C)$. This is equivalent to surjectivity of $\mu:H^0(C,\omega_{C}\otimes L_{\chi})\otimes
H^0(C,\omega_{C}\otimes L_{\chi^{-1}})\to H^0(C,\omega^{\otimes 2}_{C})$. Now take an $L_i$ from a reduced building data (\ref{reduced building data}). By \cite{Bu}, Theorem 1, this is satisfied if $\omega_{C}\otimes L_i$ is ample which follows from \cite{LO}, Lemma 5.4 and that the general curve of genus $g\geq 7$ satisfies
$\cl(C)\geq 3$.
\end{proof}
Now we treat the case of ramified Galois covers. Recall from Theorem \ref{reduced structure theorem} the reduced building data of the cover. Let $n_i$ and $d_i$ be as in \ref{reduced building data}. Then we have
\begin{proposition}
For $g\geq 2$ assume that there exists $n_i$ which is even and $d_i\geq 6$ or $n_i$ is odd and $d_i\geq 7$. Then the diffrential $d\mathcal{P}$ of the Prym map $\mathcal{P}:R(G,g,r)\to A_{p,\delta}$ at the point $((C,x_1,\dots, x_r),f:\widetilde{C}\to C)$ is injective.
\end{proposition}
\begin{proof}
By our approach based on Proposition \ref{dp}, we will show that our hypotheses imply the surjectivity of $\mu:H^0(C,\omega_{C}\otimes L_{i}^{\lfloor\frac{n_i}{2}\rfloor})\otimes
H^0(C,\omega_{C}\otimes L_i^{n_i-\lfloor\frac{n_i}{2}\rfloor})\to H^0(C,\omega^{\otimes 2}_{C}(D_i))$, where $L_i$ belongs to a reduced building data. By \cite{Bu}, Theorem 1, this is the case if both bundles $\omega_{C}\otimes L_{i}^{\lfloor\frac{n_i}{2}\rfloor}$ and $\omega_{C}\otimes L_i^{n_i-\lfloor\frac{n_i}{2}\rfloor}$ are very ample. As $\deg(\omega_{C})=2g-2$, we only need to verify that $\deg(L_i^{\lfloor\frac{n_i}{2}\rfloor})\geq 3$ and $\deg(L_i^{n_i-\lfloor\frac{n_i}{2}\rfloor})\geq 3$. But $\deg(L_i^{\lfloor\frac{n_i}{2}\rfloor})\geq \lfloor\frac{d_i}{2}\rfloor\geq 3$. Proof of
$\deg(L_i^{n_i-\lfloor\frac{n_i}{2}\rfloor})\geq 3$ is analogous.
\end{proof}
One can formulate a condition without using the structure of a reduced building data. Let $d_{\chi}=\deg L_{\chi}$. 
\begin{proposition}
For $g\geq 2$ assume that there exists a character $\chi\in G^*$ such that $d_{\chi}\geq 3$ and $d_{\chi^{-1}}\geq 3$ or that $L_{\chi}$ and $L_{\chi^{-1}}$ have non-zero global sections and $d_{\chi}+d_{\chi^{-1}}\geq 5$. Then the diffrential $d\mathcal{P}$ of the Prym map $\mathcal{P}:R(G,g,r)\to A_{p,\delta}$ at the point $((C,x_1,\dots, x_r),f:\widetilde{C}\to C)$ is injective.
\end{proposition}
\begin{proof}
The multiplication map takes the form $\mu:H^0(C,\omega_{C}\otimes L_{\chi})\otimes H^0(C,\omega_{C}\otimes L_{\chi^{-1}})\to H^0(C,\omega^{\otimes 2}_{C}(D_{\chi,\chi^{-1}}))$ and the first condtion implies that both $\omega_{C}\otimes L_{\chi}$ and $\omega_{C}\otimes L_{\chi^{-1}}$ are very ample. Now use \cite{Bu}, Theorem 1. The second condition implies that $L_{\chi}$ and $L_{\chi^{-1}}$ are globally generated and $\deg(\omega_{C}\otimes L_{\chi})+\deg(\omega_{C}\otimes L_{\chi^{-1}})\geq 4g+1$. Again \cite{Bu} implies that the multiplication map is surjective.
\end{proof}
\subsection{Prym map of metabelian covers}
In this subsection we invetigate the Prym map for metabelian Galois covers $f:\widetilde{C}\to C$. Recall the formula \ref{metabelian invariant}, $\pi_*\omega_X=\oplus (\omega_Y\otimes L^{\prime}\otimes U_{\chi})$, where $U_{\chi}=q_*\sF_{\chi}$ ($q$ is the finite map in the associated factorization  $f:
\widetilde{C}\xrightarrow{p}C_1\xrightarrow{q} C$ as at the beginning of \ref{Metabelian Galois covs}). \par Here the conditions are much more complicated because we should deal with vector bundles instead of line bundles. \par Let $\omega_C\otimes L^{\prime}$ have a non-zero global section. Furthermore suppose there exists a $U_{\chi}$ such that both $U_{\chi}$ and $U_{\chi^{-1}}$ are generated by global sections on $C$ and that 
\begin{equation} \label{condition metabelian} 
h^0(\omega_C^{\otimes 2}\otimes {L^{\prime}}^{\otimes 2}\otimes U_{\chi}\otimes U_{\chi^{-1}})\leq t(h^0(\omega_C\otimes {L^{\prime}}\otimes U_{\chi})+h^0(\omega_C\otimes {L^{\prime}}\otimes U_{\chi^{-1}}))-t^2.
\end{equation}
Where $t$ is the degree of the map $q$.
\begin{proposition}
Let $C$ satisfy the above conditions, then the differential $d\sP$ is injective at the point 
$[f:\widetilde{C}\to C]$. 
\end{proposition}
\begin{proof}
Our assumptios imply that the invertible sheaf $\omega_C\otimes L^{\prime}$ is generated by global sections. Since both $U_{\chi}$ and $U_{\chi^{-1}}$ are also globally generated by assumption, \cite{BB}, Theorem 2.1 implies that the image of the multiplication map
\[H^0(\omega_C\otimes {L^{\prime}}\otimes U_{\chi})\otimes H^0(\omega_C\otimes {L^{\prime}}\otimes U_{\chi^{-1}})\to H^0(\omega_C^{\otimes 2}\otimes {L^{\prime}}^{\otimes 2}\otimes U_{\chi}\otimes U_{\chi^{-1}})\]
has dimention $\geq t(h^0(\omega_C\otimes {L^{\prime}}\otimes U_{\chi})+h^0(\omega_C\otimes {L^{\prime}}\otimes U_{\chi^{-1}}))-t^2$. Condition \ref{condition metabelian} implies then that the multiplication is surjective and therefore the differential $d\mathcal{P}$ of the Prym map is injective at $[f:\widetilde{C}\to C]$.
\end{proof}

\end{document}